\documentclass[10pt]{amsart}

\usepackage{amsmath,amssymb,amsthm,mathtools}
\usepackage{hyperref}
\usepackage[margin=3.6cm]{geometry}

\title[]{compact Vaisman and l.c.K. manifolds with nonnegative bisectional curvature}
\author{Jiangtao Li}
\address{Department of Mathematics, University of California, San Diego, 9500 Gilman Drive, La Jolla, CA 92093, USA}
\email{jil320@ucsd.edu}

\theoremstyle{plain}
\newtheorem{theorem}{Theorem}[section]
\newtheorem{lemma}[theorem]{Lemma}

\newtheorem{corollary}[theorem]{Corollary}

\theoremstyle{definition}
\newtheorem{definition}[theorem]{Definition}
\newtheorem{example}[theorem]{Example}

\theoremstyle{remark}
\newtheorem{remark}[theorem]{Remark}

\numberwithin{equation}{section}

\DeclareMathOperator{\C}{\mathbb{C}}
\DeclareMathOperator{\Z}{\mathbb{Z}}
\DeclareMathOperator{\R}{\mathbb{R}}

\DeclareMathOperator{\CP}{\mathbb{CP}}

\DeclareMathOperator{\w}{\textnormal{\textbf{w}}}
\DeclareMathOperator{\Pic}{\text{Pic}}
\DeclareMathOperator{\rank}{\textnormal{\text{rank}}}
\DeclareMathOperator{\Aut}{\textnormal{\text{Aut}}}
\DeclareMathOperator{\Iso}{\textnormal{\text{Iso}}}
\DeclareMathOperator{\Div}{\textnormal{\text{div}}}

\newcommand{\Lie}[1]{\mathfrak{#1}}

\begin{document}

\begin{abstract}
    We study the structure of compact Vaisman and l.c.K. manifolds with nonnegative bisectional curvature in this paper. Our first main result (cf. Theorem \ref{thm: main}) is a uniformization theorem for compact Vaisman manifolds, extending Mok's uniformization theorem in the Kähler setting (cf.\cite{Mok88}). Our second result (cf. Theorem \ref{thm: lck manifolds with nonnegative bisectional curvature}) gives a structural trichotomy for compact l.c.K. manifolds with nonnegative Chern bisectional curvature: the underlying complex manifold admits a Kähler metric, admits a Vaisman metric, or its universal cover is conformally equivalent to the product of an incomplete Kähler manifold and a positive-dimensional complex Euclidean space.
\end{abstract}

\maketitle

\section{Introduction}\label{sec: introduction}

\subsection{Background} One of the landmark achievements in K\"{a}hler geometry and algebraic geometry is the resolution of the Frankel conjecture by Siu and Yau:
\begin{theorem}[\cite{SY80}]
    Any compact K\"{a}hler manifold with positive bisectional curvature is biholomorphic to the complex projective space.
\end{theorem}
\noindent A closely related conjecture, due to Hartshorne (cf.\cite{Har06}), asserts that any projective manifold with ample tangent bundle is biholomorphic to the projective space. Since ampleness is a weaker condition than the positivity of bisectional curvature, Hartshorne's conjecture is therefore more general. It was proved by Mori around the same time in his groundbreaking paper \cite{Mor79}. 

Following the work of Siu and Yau, a natural problem was to establish an analogous uniformization result for compact K\"{a}hler manifolds with nonnegative bisectional curvature. After several partial results in \cite{HSW81,Wu81,CC86}, this problem was completely settled by Mok:
\begin{theorem}[\cite{Mok88}]
    Any compact K\"{a}hler manifold with nonnegative bisectional curvature is covered by
    \begin{equation}
        (\C^r,g_0)\times(\CP^{n_1},\omega_1)\times\cdots\times(\CP^{n_k},\omega_k)\times(N_1,s_1)\times\cdots\times(N_l,s_l),
    \end{equation}
    in which $g_0$ is the flat Euclidean metric, each $\omega_i,1\leq i\leq k$ is a K\"{a}hler metric with nonnegative bisectional curvature, and each $(N_j,s_j),1\leq j\leq l$ is an irreducible compact Hermitian symmetric space of rank at least $2$.
\end{theorem}
\noindent This result is now known as Mok's uniformization theorem. The proof relies on maximum principles for K\"{a}hler--Ricci flow and Berger's classification of holonomy groups. It was later substantially simplified in \cite{Gu09}, using a more flexible maximum principle first established in \cite{BS08}.

The next question is then whether or not similar uniformization results hold for compact Hermitian manifolds and Chern bisectional curvature. By Mori's theorem \cite{Mor79}, every compact Hermitian manifold with positive Chern bisectional curvature is still biholomorphic to complex projective space so the analogue of Frankel's conjecture still holds. In contrast, for compact Hermitian manifolds with nonnegative Chern bisectional curvature, there is currently no expected general analogue of Mok's uniformization theorem. The two--dimensional case was solved in \cite[Theorem 1.9]{Yan17}. In higher dimensions, perhaps the best known results in this direction are due to \cite{Ust20}, building upon earlier work on Hermitian curvature flow; see \cite{ST11,Ust21}.

While a general uniformization theorem is out of reach, we restrict our attention to special classes of Hermitian manifolds, Vaisman manifolds and more generally, l.c.K. manifolds. These are special Hermitian manifolds introduced by Vaisman in \cite{Vai76,Vai80,Vai82}; see Definition \ref{def: lck manifolds} and Definition \ref{def: vaisman manifolds}. In the Vaisman case, we are able to establish a uniformization theorem (cf. Theorem \ref{thm: main}). For the l.c.K. case, a similar uniformization theorem is yet to be established. Instead, we prove a structure theorem (cf. Theorem \ref{thm: lck manifolds with nonnegative bisectional curvature}), which yields some interesting applications (cf. Corollary \ref{cor: lck manifolds with pi1=Z} and Corollary \ref{cor: lck metrics on non-diagonal Hopf surfaces}).

\subsection{Main theorems}
The first main theorem is a uniformization theorem for compact Vaisman manifolds with nonnegative bisectional curvature. 

\begin{theorem}\label{thm: main}
    Let $(M,g)$ be a compact Vaisman manifold with nonnegative bisectional curvature. Then its universal cover $\widetilde{M}$ is biholomorphic to a quasi--regular K\"{a}hler cone (without the vertex), whose leaf space has the following form
    \begin{equation}\label{eq: leaf space}
        \CP(\w_1)\times\cdots\times\CP(\w_k)\times N_1\times\cdots\times N_l,
    \end{equation}
    where $\CP(\w_i), 1\leq i\leq k$, are weighted projective spaces and $N_j, 1\leq j\leq l$, are underlying complex manifolds of irreducible compact Hermitian symmetric spaces whose rank is at least $2$.

    Conversely, any simply connected quasi--regular K\"{a}hler cone whose leaf space is as in \eqref{eq: leaf space} is biholomorphic to the universal cover of some compact Vaisman manifold with nonnegative bisectional curvature.
\end{theorem}

In dimension $2$, the above theorem yields the following corollary:

\begin{corollary}
    A compact Vaisman surface with nonnegative bisectional curvature is a Hopf surface, i.e. covered by $\C^2\setminus\{(0,0)\}$.
\end{corollary}
\begin{remark}
    In fact, there is a much stronger result. By combining \cite[Proposition 1.10]{Yan17} and \cite[Theorem 1]{Bel00}, one deduces that a compact complex surface admits a Vaisman metric with nonnegative bisectional curvature only if it is a Hopf surface of class 1, i.e., a quotient of some diagonal Hopf surface.
\end{remark}

The second main theorem is a structure theorem for compact l.c.K. manifolds with nonnegative bisectional curvature. 
\begin{theorem}\label{thm: lck manifolds with nonnegative bisectional curvature}
    If $(M,g)$ is a compact l.c.K. manifold with nonnegative bisectional curvature, then at least one of the following holds:
    \begin{enumerate}
        \item $M$ admits a K\"{a}hler metric;
        \item $M$ admits a Vaisman metric;
        \item The universal cover $(M,\widetilde{g})$ is conformally equivalent to $(L,g_L)\times(\C^k, g_0)$, where $(L,g_L)$ is an incomplete K\"{a}hler manifold and $(\C^k, g_0), k>0$, is a complex Euclidean space.
    \end{enumerate}
\end{theorem}
\begin{remark}
    Case (3) in the theorem can be ruled out by requiring $\pi_1(M)=\Z$; see Corollary \ref{cor: lck manifolds with pi1=Z}. 
\end{remark}

\subsection{Outline of the paper} 
\begin{enumerate}
    \item Section \ref{sec: preliminaries} reviews basic notions and results from Sasaki geometry and the theory of compact Vaisman manifolds, which will be needed for the proofs of Theorem \ref{thm: main} and Theorem \ref{thm: lck manifolds with nonnegative bisectional curvature}.
    
    \item In Section \ref{sec: structure of the universal cover}, we prove the first part of Theorem \ref{thm: main}; the second part is proved in Section \ref{sec: vaisman metrics}.
    
    The idea for the first part is as follows: first, it is well known that compact Vaisman manifolds are covered by K\"{a}hler cones (cf. Theorem \ref{thm: structure theorem}). The nonnegativity of bisectional curvature is then equivalent to a certain curvature condition on the K\"{a}hler cone (cf. Corollary \ref{cor: reduction to kahler cone}).We then apply the Sasaki analogue of Mok's theorem, proved by He and Sun (cf. Theorem \ref{thm: He-Sun}), to obtain the desired characterization of the universal cover. 

    Section \ref{sec: vaisman metrics} provides an explicit construction of Vaisman metrics with nonnegative bisectional curvature on all K\"{a}hler cones in Theorem \ref{thm: main}. 
    
    \item Section \ref{sec: lck manifolds} deals with the l.c.K. case. In the Subsection \ref{subsec: proof}, we prove Theorem \ref{thm: lck manifolds with nonnegative bisectional curvature}. The key ingredient is a theorem in \cite{Ust20}, which shows that if $M$ is non--K\"{a}hler, the universal cover of $(M,g)$ must possess a nontrivial complex Lie group action. A detailed study of this group action yields (2) and (3) in Theorem \ref{thm: lck manifolds with nonnegative bisectional curvature}. 
    
    Two corollaries (Corollary \ref{cor: lck manifolds with pi1=Z} and Corollary \ref{cor: lck metrics on non-diagonal Hopf surfaces}) are given in Subsection \ref{subsec:applications}. We shall also discuss possible improvements and questions in Subsection \ref{subsec: questions}.
\end{enumerate}

\section{Preliminaries}\label{sec: preliminaries}

\subsection{Sasaki geometry and K\"{a}hler cones}

This subsection reviews some of the fundamental concepts and results regarding compact Sasaki manifolds and their associated K\"{a}hler cones. For a comprehensive introduction to these topics, readers are referred to the monograph by Boyer and Galicki (cf. \cite{BG07}).

Let $(X,g_X)$ be a Riemannian manifold. The cone over it is $C(X)=\R_+\times X$ equipped with the cone metric $dr^2+r^2g_X$. Throughout this paper, all cones are understood to have their vertices removed.

A \textit{Sasaki structure} on $(X,g_X)$ is an integrable complex structure $J$ on $C(X)$ such that 
\begin{enumerate}
    \item $(C(X),J,dr^2+r^2g_X)$ is K\"{a}hler.
    \item $\mathcal{L}_{r\frac{\partial}{\partial r}}J=0$, equivalently, $r\frac{\partial}{\partial r}$ is a real holomorphic vector field.
\end{enumerate}
Under these conditions, $(X,g_X,J)$ is called a \textit{Sasaki manifold}, $(C(X),J,dr^2+r^2g_X)$ is referred to as the associated \textit{K\"{a}hler cone}.

For a given K\"{a}hler cone $(C(X),J,dr^2+r^2g_X)$, the vector field $\xi=Jr\frac{\partial}{\partial r}$ is tangent to the Sasaki link $X\cong \{1\}\times X$ and it is known as the \textit{Reeb vector field}. The Reeb vector field is a Killing vector field of constant unit length. The associated $1$-form $\eta(\cdot)=g_X(\xi,\cdot)$ is called the \textit{contact form}. Let $\mathcal{D}$ denote the distribution defined by the orthogonal complement to $\xi$ on $X$ and $g^T$ be the restriction of the metric $g_X$ on $\mathcal{D}$. The metric $g_X$ can thus be decomposed as
\begin{equation}
    g_X=\eta\otimes\eta+g^T.
\end{equation}
Furthermore, the K\"{a}hler form $\omega$ of the cone metric $dr^2+r^2g_X$ is given by
\begin{equation}
    \omega=rdr\wedge\eta+r^2\omega^T=\frac{\sqrt{-1}}{2}\partial\Bar{\partial}r^2,
\end{equation}
where $\omega^T(\cdot,\cdot)=g^T(J\cdot,\cdot)$ represents the \textit{transverse K\"{a}hler form}. This formula demonstrates that the K\"{a}hler cone metric is entirely determined by its radial function $r$. 

The Reeb field generates a one-parameter subgroup of Sasaki automorphisms on $X$, inducing a 1-dimensional foliation $\mathcal{F}_\xi$ on the manifold. K\"{a}hler cones are classified into two primary types based on the closedness of leaves in $\mathcal{F}_\xi$:
\begin{enumerate}
    \item \textit{irregular}: there is a non-closed leaf.
    \item \textit{quasi-regular}: all leaves are closed. Equivalently, $\xi$ generates an $S^1$ action on $X$ via Sasaki automorphisms. In this case, the quotient of $X$ by the Reeb flow, $X/\mathcal{F}_\xi$, is a complex orbifold, and the transverse K\"{a}hler form descends to an orbifold K\"{a}hler form on it. This K\"{a}hler orbifold is called the \emph{leaf space} of the K\"{a}hler cone.
    
    If $\xi$ induces a free $S^1$ action, the K\"{a}hler cone is called \textit{regular}. In this case, the leaf space becomes a K\"{a}hler manifold.
\end{enumerate}

Two fundamental deformations can be applied to a given K\"{a}hler cone:
\begin{enumerate}
    \item A \textit{type I deformation} preserves the complex structure $J$ and varies the Reeb vector field $\xi$ and the metric.
    \item A \textit{type II deformation} preserves both $J$ and $\xi$ but deforms the metric.
\end{enumerate}
Further details on these deformations can be found in the literature (cf. \cite{MSY08,FOW09}).
This paper exclusively utilizes a specific class of type I deformation known as a \textit{$\mathcal{D}$-homothety}: Specifically, if $\omega=\frac{\sqrt{-1}}{2}\partial\Bar{\partial}r^2$ is a K\"{a}hler cone, a $\mathcal{D}$-homothety produces new K\"{a}hler cone structure given by $\omega_a=\frac{\sqrt{-1}}{2}\partial\Bar{\partial}r^{2a}, a>0$.

We give two classes of examples here, which are important for our later arguments.

\begin{example}[Weighted Sasaki spheres]
    Consider the standard Euclidean space $((\C^n)^\times, g_0)$. Following \cite{HS15,HS16}, a \textit{simple deformation} is a K\"{a}hler cone structure obtained from $g_0$ by a type I deformation followed by a type II deformation. For a simply deformed K\"{a}hler cone from $g_0$, the associated Sasaki link is called a \textit{weighted Sasaki sphere}, which is a sphere $S^{2n-1}$ with the deformed Sasaki structure.
\end{example}

\begin{example}[Negative line bundles over projective manifolds]\label{exm: quasi regular cones}
    Let $L\to N$ be a positive line bundle over a projective manifold $N$ and $(L^{-1})^\times$ be the complement of the zero section within $L^{-1}$. Choose a Hermitian metric $h$ on $L^{-1}$. It induces a radial function $r_h(\zeta)=h(\zeta,\Bar{\zeta})^{1/2}$ on $(L^{-1})^\times$. Then $\omega_h=\frac{\sqrt{-1}}{2}\partial\Bar{\partial}r^2_h$ defines a K\"{a}hler cone structure on $(L^{-1})^\times$. The K\"{a}hler cone is regular because the Reeb flow is the free $S^1$ action on the fibers of $L^{-1}$.
    
    The transverse K\"{a}hler form $\omega_h^T$ of $((L^{-1})^\times,\omega_h)$ descends to $N$ and it is equal to $\frac{\sqrt{-1}}{2}\partial\Bar{\partial}\log h=-\frac{\sqrt{-1}}{2}R_h$, where $R_h$ is the curvature form of $h$. This in particular implies that $[\omega_h^T]=\pi c_1(L)$.

    This construction also works for a positive orbifold line bundle $L$ over a projective orbifold $N$. In this broader context, the construction yields a quasi-regular K\"{a}hler cone. Furthermore, the transverse K\"{a}hler form $\omega_h^T$ is an orbifold K\"{a}hler form and $[\omega^T_h] = \pi c_1^{orb}(L)\in H^2_{orb}(X,\Z)$. We will mainly deal with the case when $N$ is a product of weighted projective spaces in this note.
\end{example}

We shall need the following structure theorem of quasi-regular K\"{a}hler cones, which asserts that essentially all such K\"{a}hler cones arise from the construction in the above example. 

\begin{theorem}[\text{\cite[Theorem 1.5]{Spa10}}]\label{thm: quasi-regular kahler cones}
    Let $(C(X),J,g_C)$ be a quasi-regular K\"{a}hler cone. There exists a positive orbifold line bundle $L\to X/\mathcal{F}_\xi$ and an orbifold Hermitian metric $h$ on $L^{-1}$ such that $(C(X),J,g_C)$ is $\mathcal{D}$ homothetic to $((L^{-1})^\times,\omega_h)$. Equivalently, it is isomorphic to 
    \begin{equation}
        \left((L^{-1})^\times,\frac{\sqrt{-1}}{2}\partial\Bar{\partial}r_h^{2a}\right), \quad a>0.
    \end{equation}
\end{theorem}
\begin{remark}
    There is a subtlety in this theorem. In Theorem 1.5 of \cite{Spa10}, the author normalizes the K\"{a}hler cone by assuming that $\xi$ generates a $U(1)$ action. The above formulation doesn't impose this normalization. This is the reason why the $\mathcal{D}$ homothety shows up in the above statement.
\end{remark}

Sasaki geometry is frequently analyzed through the associated transverse K\"{a}hler geometry, which mirrors compact K\"{a}hler geometry. Consequently, there are natural analogues of the known results about compact K\"{a}hler manifolds for K\"{a}hler cones with compact Sasaki links. Among them is the following generalization of Mok's uniformization theorem (cf. \cite{Mok88}) shown by He and Sun (cf. \cite{HS15}):

\begin{theorem}\label{thm: He-Sun}
    If a compact Sasaki manifold $(X,g_X,J)$ has nonnegative transverse bisectional curvature, then
    \begin{enumerate}
        \item If the Reeb flow is irregular, then the universal cover of $(X,g_X,J)$ is isomorphic to a weighted Sasaki sphere. 
        \item If the Reeb flow is quasi-regular, then the universal cover of the leaf space $(X/\mathcal{F}_\xi, \omega^T)$ is isomorphic to 
        \begin{equation}\label{eq: quasiregular case}
            (\C^r,g_0)\times\prod_{i=1}^k(\CP(\w_i),\omega_i)\times\prod_{j=1}^l(N_j,s_j),
        \end{equation}
        where each $\omega_i, 1\leq i\leq k$ is a K\"{a}hler metric on $\CP(\w_i)$ with nonnegative bisectional curvature and each $(N_j,s_j), 1\leq j\leq l$ is an irreducible compact Hermitian symmetric space and $\rank N_j\ge 2$.
    \end{enumerate}
\end{theorem}

\subsection{Structure of compact Vaisman manifolds}

We review basics of locally conformally K\"{a}hler manifolds and Vaisman manifolds in this subsection. Standard references include, for example, \cite{Vai80,Vai82, OV24}. 

\begin{definition}\label{def: lck manifolds}
    Let $(M,g)$ be a Hermitian manifold. It is \textit{locally conformally K\"{a}hler} (or \textit{l.c.K.} in short) if there is a 1-form $\theta$ such that 
    \begin{equation}
        d\omega = \omega\wedge\theta, \quad d\theta = 0
    \end{equation}
    where $\omega$ is the fundamental form of $g$. $\theta$ is called the \textit{Lee form} of $(M,g)$.
\end{definition}

\begin{definition}\label{def: vaisman manifolds}
    An l.c.K. manifold $(M,g)$ is \textit{Vaisman} if its Lee form is nonzero and parallel, i.e., $\nabla \theta=0$, where $\nabla$ is the Levi-Civita connection of $g$. 
\end{definition}
\begin{remark}
    This definition implies that $|\theta|$ is a nonzero constant. In this paper, we will always scale properly so that $|\theta|\equiv \sqrt{2}$. 
\end{remark}

Vaisman manifolds were initially introduced in \cite{Vai80,Vai82} as a generalization of diagonal Hopf manifolds. They form one of the major families of non--K\"{a}hler manifolds. They are closely related to Kähler cones. In fact, the following structure theorem holds:

\begin{theorem}[\cite{GOP05}]\label{thm: structure theorem}
    Let $(M,g)$ be a complete Vaisman manifold with exact Lee form $\theta$. There exists a K\"{a}hler cone $(C(X),J,dr^2+r^2g_X)$ with complete Sasaki link $(X,g_X)$ such that $(M,g)$ is biholomorphic and isometric to $(C(X),g^V)$, where $g^V=\frac{dr^2}{r^2}+g_X$ is a Vaisman metric on $C(X)$. 
    
    In particular, this holds for simply conneced complete Vaisman manifold.
\end{theorem}

\section{Structure of the universal cover}\label{sec: structure of the universal cover}

In this section, we provide a proof of the first statement of Theorem \ref{thm: main}. 

Consider a K\"{a}hler cone $(C(X),J,g)$. We compute the Chern curvature tensor $R^V$ of the associated Vaisman metric $g^V$ on $C(X)$. Recall that a set of local transverse holomorphic coordinates on $C(X)$ consists of holomorphic functions $z^1,\cdots,z^{n-1}$ such that
$\mathcal{L}_{\xi}z^i=\mathcal{L}_{r\frac{\partial}{\partial r}}z^i=0$ and $dz^i$'s are linearly independent everywhere. Furthermore, there exists a radial holomorphic coordinate $w$ such that $w,z^1,\cdots,z^{n-1}$ form a system of local holomorphic coordinates on $C(X)$. In these coordinates the transverse K\"{a}hler form $\omega^T$ depends only on the transverse coordinates $z^1,\cdots,z^{n-1}$. 

The following local expression for the K\"{a}hler cone metric is standard:
\begin{lemma}\label{lem: local holomorphic coordinates}
    There exists a K\"{a}hler potential $K(z^1,\cdots,z^{n-1})$ of $\omega^T$ such that $dK=0$ at a given point $p\in C(X)$ and 
    \begin{equation}\label{eq: kahler cone metric}
        \omega=\sqrt{-1}r^2\left(\frac{1}{2}dw\wedge d\Bar{w}+dw\wedge\Bar{\partial}K+\partial K\wedge d\Bar{w}+2\partial K\wedge\Bar{\partial K}+\partial\Bar{\partial}K\right).
    \end{equation}
    In particular,
    \begin{equation}\label{eq: kahler cone metric at one point}
        \omega=\sqrt{-1}r^2\left(\frac{1}{2}dw\wedge d\Bar{w}+\partial\Bar{\partial}K\right)
    \end{equation}
    at $p$.
\end{lemma}

The curvature form $R^V$ can be expressed explicitly in terms of the transverse K\"{a}hler metric $g^T$ and the transverse curvatures, as shown in the lemma below:

\begin{lemma}\label{lem: chern curvature}
    Under the choice of local holomorphic coordinates in Lemma \ref{lem: local holomorphic coordinates}, the Chern curvature tensor $R^V$ of the Vaisman metric $g^V$ at the given point $p$ has the following local components:
    \begin{equation}
        \begin{aligned}
            R^V_{i\Bar{j}k\Bar{l}}&= R^T_{i\Bar{j}k\Bar{l}}-2g^T_{i\Bar{l}}g^T_{k\Bar{j}},\\
            R^V_{i\Bar{j}w\Bar{w}}&= g^T_{i\Bar{j}},\\
            R^V_{i\Bar{j}w\Bar{l}}&=0,\\
            R^V_{w\Bar{j}k\Bar{l}}&=R^V_{w\Bar{j}w\Bar{l}}=R^V_{w\Bar{j}w\Bar{w}}=0,\\
            R^V_{w\Bar{w}k\Bar{l}}&=R^V_{w\Bar{w}k\Bar{w}}=R^V_{w\Bar{w}w\Bar{w}}=0.
        \end{aligned}
    \end{equation}
\end{lemma}
\begin{proof}
    We present the computation of $(R_V)_{i\Bar{j}k\Bar{l}}$. The other components can be calculated in the same manner. Recall the general formula for the Chern curvature tensor of Hermitian metrics, 
    \[R^V_{i\Bar{j}k\Bar{l}}=(g^V)^{\alpha\Bar{\beta}}\frac{\partial g^V_{k\Bar{\beta}}}{\partial z^i}\frac{\partial g^V_{\alpha\Bar{l}}}{\partial \Bar{z}^j}-\frac{\partial^2 g^V_{k\Bar{l}}}{\partial z^i\Bar{z}^j},\]
    where the indices $\alpha$ and $\beta$ run over all coordinates $w,z^1,\cdots,z^{n-1}$. By \eqref{eq: kahler cone metric} and \eqref{eq: 
    kahler cone metric at one point}, one deduces that
    \begin{equation}\label{eq: vaisman metric}
        \begin{aligned}
            &(g^V)_{w\Bar{w}}=\frac{1}{2},\\
            &(g^V)_{w\Bar{i}}=\frac{\partial K}{\partial \bar{z}^i},\\
            &(g^V)_{i\Bar{j}}=2\frac{\partial K}{\partial z^i}\frac{\partial K}{\partial \Bar{z}^j}+\frac{\partial^2 K}{\partial z^i\partial \Bar{z}^j}.
        \end{aligned}
    \end{equation}
    in an open neighborhood of $p$ and 
    \begin{equation}\label{eq: vaisman metric at one point}
        g_{w\Bar{w}}=\frac{1}{2}, \quad g_{w\Bar{i}}=0, \quad g_{i\Bar{j}}=\frac{\partial^2 K}{\partial z^i\partial\Bar{z}^j}
    \end{equation}
    at $p$. The curvature then follows directly from \eqref{eq: vaisman metric} and \eqref{eq: vaisman metric at one point}.
    \begin{align*}
        R^V_{i\Bar{j}k\Bar{l}}=&(g^V)^{w\Bar{w}}\frac{\partial g^V_{k\Bar{w}}}{\partial z^i}\frac{\partial g^V_{w\Bar{l}}}{\partial \Bar{z}^j}+(g^V)^{pq}\frac{\partial g^V_{k\Bar{q}}}{\partial z^i}\frac{\partial g^V_{p\Bar{l}}}{\partial \Bar{z}^j}-\frac{\partial^2 g^V_{k\Bar{l}}}{\partial z^i\Bar{z}^j}\\
        =&2\frac{\partial^2 K}{\partial z^i\partial z^k}\frac{\partial^2 K}{\partial \Bar{z}^j\partial \Bar{z}^l}+g^{p\Bar{q}}\frac{\partial}{\partial z^i}\left(\frac{\partial^2 K}{\partial z^k\partial \Bar{z}^q}\right)\frac{\partial}{\partial \Bar{z}^j}\left(\frac{\partial^2 K}{\partial z^p\Bar{z}^{l}}\right)\\
        &-\frac{\partial^2}{\partial z^i\Bar{z}^j}\left(2\frac{\partial K}{\partial z^k}\frac{\partial K}{\partial \Bar{z}^l}+\frac{\partial^2 K}{\partial z^k\partial \Bar{z}^l}\right)\\
        =&(g^T)^{p\Bar{q}}\frac{\partial g^T_{k\Bar{q}}}{\partial z^i}\frac{\partial g^T_{p\Bar{l}}}{\partial \Bar{z}^j}-2\frac{\partial^2 K}{\partial z^i\partial \Bar{z}^l}\frac{\partial^2 K}{\partial \Bar{z}^j\partial z^k}-\frac{\partial^2g^T_{k\Bar{l}}}{\partial z^i\partial \Bar{z}^j}\\
        =&\left((g^T)^{p\Bar{q}}\frac{\partial g^T_{k\Bar{q}}}{\partial z^i}\frac{\partial g^T_{p\Bar{l}}}{\partial \Bar{z}^j}-\frac{\partial^2g^T_{k\Bar{l}}}{\partial z^i\partial \Bar{z}^j}\right)-2g^T_{i\Bar{l}}g^T_{k\Bar{j}}\\
        =& R^T_{i\Bar{j}k\Bar{l}}-2g^T_{i\Bar{l}}g^T_{k\Bar{j}}.
    \end{align*}
\end{proof}

Based on this lemma, the condition of nonnegative bisectional curvature for $g^V$ is equivalent to a curvature condition on the transverse bisectional curvature.

\begin{corollary}\label{cor: reduction to kahler cone}
    Let $(C(X),J,g)$ be an $n$ dimensional K\"{a}hler cone. $g^V$ has nonnegative bisectional curvature if and only if $g^T$ satisfies the following curvature condition:
    \begin{equation}\label{eq: key curvature condition}
        R^T_{i\Bar{j}k\Bar{l}}\xi^i\Bar{\xi^j}\zeta^k\Bar{\zeta^l}\ge 2g^T_{i\Bar{l}}g^T_{k\Bar{j}}\xi^i\Bar{\xi^j}\zeta^k\Bar{\zeta^l}, \quad \forall\xi, \zeta\in \C^{n-1}.
    \end{equation}
\end{corollary}

By combining this with Theorem \ref{thm: structure theorem}, we may obtain the following equivalent characterization of compact Vaisman manifolds with nonnegative bisectional curvature:
\begin{corollary}
    A compact Vaisman manifold $(M,g)$ has nonnegative bisectional curvature if and only if its universal cover $(\widetilde{M},\widetilde{g})$ is biholomorphic and isometric to $(C(X),J,\frac{dr^2}{r^2}+g_X)$, where $(C(X),J,dr^2+r^2g_X)$ is a K\"{a}hler cone satisfying \eqref{eq: key curvature condition}.
\end{corollary}

Therefore, it remains to prove that any simply connected K\"{a}hler cone satisfying \eqref{eq: key curvature condition} is one of the cones described in Theorem \ref{thm: main}. To that end, we first show that the Sasaki link is compact. Indeed, \eqref{eq: key curvature condition} implies that the transverse holomorphic sectional curvature is bounded below by $2$. Compactness then follows from Corollary \ref{cor: compactness}. 

Given the compactness of the Sasaki link and the fact that \eqref{eq: key curvature condition} implies nonnegative transverse bisectional curvature, we are able to invoke Theorem \ref{thm: He-Sun} to finish the proof. We consider the irregular and quasi--regular cases separately.

\subsection{Irregular case} The irregular case of Theorem \ref{thm: He-Sun} implies that $(X,g_X,J)$ is a weighted Sasaki sphere. Therefore, the K\"{a}hler cone is biholomorphic to $(\C^n)^\times$, which admits a regular K\"{a}hler cone structure whose leaf space is $\CP^{n-1}$.

\subsection{Quasi--regular case} If $(C(X),J,dr^2+r^2g_X)$ is quasi--regular, $X\to X/\mathcal{F}_\xi$ is an orbifold principle $S^1$ bundle. The associated long exact sequence
\begin{equation}
    \cdots \to \pi_1(S^1)\to \pi_1(X) \to \pi^{orb}_1(X/\mathcal{F}_\xi)\to \pi_0(S^1)\to \cdots 
\end{equation}
implies $\pi^{orb}_1(X/\mathcal{F}_\xi) = 1$ since $X$ is simply connected. By (2) of Theorem \ref{thm: He-Sun}, $X/\mathcal{F}_\xi$ is isomorphic to the product in \eqref{eq: quasiregular case}.

The positivity of transverse holomorphic sectional curvature rules out the Euclidean factor in the product. Therefore,
\begin{equation*}
    X/\mathcal{F}_\xi \cong \CP(\w_1)\times\cdots\times\CP(\w_k)\times N_1\times\cdots\times N_l.
\end{equation*} 
This completes the proof.

\section{Nonnegatively curved Vaisman metrics on K\"{a}hler cones}\label{sec: vaisman metrics}

We prove the second part of Theorem \ref{thm: main}: any simply connected quasi--regular K\"{a}hler cone whose leaf space is as in \eqref{eq: leaf space} covers a compact Vaisman manifold with nonnegative bisectional curvature. 

Suppose $(C(X), J, dr^2+r^2g_X)$ is such a K\"{a}hler cone. We first construct a new K\"{a}hler cone metric with curvature condition
\begin{equation}\label{eq: curvature condition}
    R^T(\xi,\Bar{\xi},\zeta,\Bar{\zeta})\ge c|g^T(\xi,\Bar{\zeta})|^2,
\end{equation}
where $c$ is a positive constant. For this purpose, recall the structure theorem of quasi--regular K\"{a}hler cones (cf. Theorem \ref{thm: quasi-regular kahler cones}), which implies that $C(X)$ is biholomorphic to $(L^{-1})^\times$, where $L$ is a positive orbifold line bundle over $X/\mathcal{F}_\xi$. Since $X/\mathcal{F}_\xi$ is given as in \eqref{eq: leaf space}, 
\begin{equation}
    L = \mathcal{O}_{\CP(\w_1)}(n_1)\otimes\cdots\otimes\mathcal{O}_{\CP(\w_k)}(n_k)\otimes L_1^{m_1}\otimes\cdots\otimes L_l^{m_l}.
\end{equation}
where $L_j$ is the positive generator of $\Pic(N_j)$ and $n_i, m_j$ are all positive integers. We will firstly find an orbifold K\"{a}hler metric in $\pi c^{orb}_1(L)$ satisfying \eqref{eq: curvature condition}.

\begin{lemma}\label{lem: weighted projective spaces}
    Let $\CP(\w)$ be a weighted projective space. There exists an orbifold K\"{a}hler metric and a constant $c>0$ such that \eqref{eq: curvature condition} holds. 
\end{lemma}
\begin{proof}
    It follows from \cite{HS16} that the orbifold K\"{a}hler-Ricci soliton metric $\omega$ on $\CP(\w)$ has positive bisectional curvature. Therefore,
    \[c=\inf_{|\xi|=1,|\zeta|=1}R(\xi,\Bar{\xi},\zeta,\Bar{\zeta})>0.\]
    It follows that for any unit holomorphic tangent vectors $\xi$ and $\zeta$, $R(\xi,\Bar{\xi},\zeta,\Bar{\zeta})\ge c|g(\xi,\Bar{\zeta})|^2$. 
\end{proof}

The condition \eqref{eq: curvature condition} also holds for irreducible compact Hermitian symmetric spaces. The lemma below provides a simple proof of this fact.

\begin{lemma}\label{lem: symmetric spaces}
    Let $(N,s)$ be an irreducible compact Hermitian symmetric space. Then there is a constant $c>0$ such that \eqref{eq: curvature condition} holds.
\end{lemma}
\begin{proof}
    The proof is based on a modification of the calculation in \cite{OT81} and we will use the notations in \cite{OT81}. Let $N=G/K_i$, $\rank N=r$ and $\{\delta_1,\cdots,\delta_r\}$ be a maximal set of strongly orthogonal roots in $\Delta_i$. Then the subspace $\Lie{a}\subset\Lie{m}_i$ spanned by $A_{\delta_1},\cdots,A_{\delta_r}$ is a maximal abelian subspace in $\Lie{m}_i$. Let $X,Y\in T_oN\cong\Lie{m}_i$ be any pair of real tangent vectors to $N$. As in \cite{OT81}, we may assume, without loss of generality, that $X\in\Lie{a}$ and 
   \[X=\sum_{j=1}^rx_jA_{\delta_j}.\]
   Moreover, let $Y=Y_1+Y_2$, where $Y_1\in\Lie{a}\oplus J\Lie{a}$ and $Y_2\in (\Lie{a}\oplus J\Lie{a})^\perp$ and suppose that
   \[Y_1=\sum_{j=1}^r\left(y_jA_{\delta_j}+z_jB_{\delta_j}\right).\]
   Since $R(X,JX)$ preserves both $\Lie{a}\oplus J\Lie{a}$ and $(\Lie{a}\oplus J \Lie{a})^\perp$ and $(N,s)$ has nonnegative bisectional curvature,
   \begin{equation}\label{eq: first inequality}
        \begin{aligned}
            R(X,JX,Y,JY)&=R(X,JX,Y_1,JY_1)+R(X,JX,Y_2,JY_2)\\
                    & \ge R(X,JX,Y_1,JY_1).
        \end{aligned}
    \end{equation}
   It remains to estimate $R(X,JX,Y_1,JY_1)$. By the assumptions on $X$ and $Y_1$,
    \begin{equation}
        \begin{aligned}
         &R(X,JX,Y_1,JY_1)\\
        =&s(-[[X,JX],Y_1],JY_1)\\
        =&s\left(-\left[\left[\sum_{j=1}^rx_jA_{\delta_j},\sum_{j=1}^rx_jB_{\delta_j}\right],Y_1\right],JY_1\right)\\
        =&\sum_{j=1}^rx_j^2s\left(\left[-2\sqrt{-1}H_{\delta_j},Y_1\right],JY_1\right)\\
        =&\sum_{j=1}^rx_j^2s\left(\left[-2\sqrt{-1}H_{\delta_j},\sum_{k=1}^r\left(y_kA_{\delta_k}+z_kB_{\delta_k}\right)\right],\sum_{k=1}^r\left(y_kB_{\delta_k}+z_kA_{\delta_k}\right)\right)\\
        =&\sum_{j=1}^rx_j^2s\left(2y_j(\delta_j,\delta_j)B_{\delta_j}+2z_j(\delta_j,\delta_j)A_{\delta_j},\sum_{k=1}^r\left(y_kB_{\delta_k}+z_kA_{\delta_k}\right)\right)\\
        =&\sum_{j=1}^r2x_j^2y_j^2(\delta_j,\delta_j)\|B_{\delta_j}\|^2_s+2x_j^2z_j^2(\delta_j,\delta_j)\|A_{\delta_j}\|_s^2.
       \end{aligned}
    \end{equation}
    By the property of strongly orthogonal roots, $(\delta_1,\delta_1)=\cdots=(\delta_r,\delta_r)=l_1$, where $l_1$ is a constant depending on the space $N$. Furthermore, $\|A_{\delta_1}\|_s^2=\cdots=\|A_{\delta_r}\|^2_s=\|B_{\delta_1}\|_s^2=\cdots=\|B_{\delta_r}\|_s^2=l_2$.
    Hence, 
    \begin{equation}\label{eq: second inequality}
        \begin{aligned}
         &R(X,JX,Y_1,JY_1)\\
        =&2l_1l_2\sum_{j=1}^rx_j^2y_j^2+x_j^2z_j^2\\
        \ge&\frac{2l_1}{rl_2}\left(\left(\sum_{j=1}^rx_jy_j\|A_{\delta_j}\|^2_s\right)^2+\left(\sum_{j=1}^rx_jz_j\|A_{\delta_j}\|^2_s\right)^2\right)\\
        =&\frac{2l_1}{rl_2}\left(s(X,Y)^2+s(X,JY)^2\right)
       \end{aligned}
    \end{equation}
    It is easy to see that this implies 
    \[R(\xi,\Bar{\xi},\zeta,\Bar{\zeta})\ge \frac{2l_1}{rl_2}|s(\xi,\Bar{\zeta})|^2.\]
\end{proof}
After suitably rescaling the metrics constructed in the above two lemmas and then taking their product, we may construct an orbifold metric $\omega^T$ on $X/\mathcal{F}_\xi$ in $\pi c^{orb}_1(L)$ with curvature condition \eqref{eq: curvature condition}.

Since $\omega^T \in \pi c^{orb}_1(L)$, there exists an orbifold Hermitian metric $h$ on $L^{-1}$ whose curvature form is $-2\sqrt{-1}\omega^T$. Let $r_h$ be the radial function defined by $h$. Following Example \ref{exm: quasi regular cones}, $\frac{\sqrt{-1}}{2}\partial\Bar{\partial}r^2_h$ is a K\"{a}hler cone metric whose transverse K\"{a}hler metric is $\omega^T$. Thus this K\"{a}hler cone metric satisfies \eqref{eq: curvature condition}.

Next, we apply a $\mathcal{D}$ homothety and consider the deformed K\"{a}hler cone metric
\[\omega = \frac{\sqrt{-1}}{2}\partial\Bar{\partial}r^{2a}_h, \quad a = \frac{2}{c},\]
whose transverse K\"{a}hler metric is $a\omega^T$. Then $\omega$ satisfies the curvature condition \eqref{eq: key curvature condition}. It follows from Corollary \ref{cor: reduction to kahler cone} that the associated Vaisman metric $g^V$ has nonnegative Chern bisectional curvature.

Finally, following the construction in \cite[\S 1]{OV03}, there is a cocompact, free, isometric $\Z$--action on the Vaisman manifold $(C(X), g^V)$. This shows that $C(X)$ covers a compact Vaisman manifold with nonnegative bisectional curvature.

\section{Compact l.c.K. manifolds with nonnegative bisectional curvature}\label{sec: lck manifolds}

In this section we study the more general case of compact l.c.K. manifolds with nonnegative bisectional curvature. 

\subsection{Proof of Theorem \ref{thm: lck manifolds with nonnegative bisectional curvature}}\label{subsec: proof}
Assume that $(M,J)$ does not admit a K\"{a}hler metric. We show that either (2) or (3) in Theorem \ref{thm: lck manifolds with nonnegative bisectional curvature} holds.

Because $(M,J)$ is non--Fano, \cite[Corollary 5.4]{Ust20} implies that the universal cover $\widetilde{M}$ admits an almost free, holomorphic and isometric action by a nontrivial simply connected complex Lie group $G$. Suppose that this action is given by the group homomorphism
\begin{equation}
    \rho: G \to \Aut(\widetilde{M})\cap\Iso(\widetilde{M},\widetilde{g}).
\end{equation}
Let $\Lie{g}$ be the Lie algebra of $G$, which is then identified with a Lie algebra of real holomorphic Killing vector fields. 

Since $(M,J,g)$ is l.c.K., there is a K\"{a}hler metric $\widetilde{g}_K$ on $\widetilde{M}$, such that 
\begin{equation}
    \widetilde{g}_K = e^f\widetilde{g}
\end{equation}
and $G$ acts conformally with respect to $\widetilde{g}_K$. Because $\widetilde{g}_K$ is K\"{a}hler, $G$ acts by homotheties with respect to $\widetilde{g}_K$. Therefore, there is a group homomorphism $\phi: G\to \R_+$, such that 
\begin{equation}
    \rho(x)^*\widetilde{g}_K = \phi(x)\widetilde{g}_K.
\end{equation}
We distinguish two cases according to whether $\phi$ is trivial. 

\subsubsection{$\phi$ is nontrivial}\label{subsubsection: non-trivial}

In this case, we prove alternative (2) in Theorem \ref{thm: lck manifolds with nonnegative bisectional curvature}. First, because of the nontriviality of $\phi$, there are $\xi\in\Lie{g}, E=-J\xi\in\Lie{g}$ such that 
\begin{equation}\label{eq: Lie derivatives}
    \mathcal{L}_\xi\widetilde{g}_K = 0, \quad \mathcal{L}_E\widetilde{g}_K = 2\widetilde{g}_K.
\end{equation}
The following lemma constructs a K\"{a}hler cone structure on $(\widetilde{M},\widetilde{g}_K)$ by these vector fields.
\begin{lemma}
    $(\widetilde{M},\widetilde{g}_K)$ is biholomorphic and isometric to a K\"{a}hler cone $(\R_+\times S, dr^2+r^2g_S)$ and the associated Reeb field is $\xi$.
\end{lemma}
\begin{proof}
    Because $\mathcal{L}_\xi\widetilde{g}_K=0$, for any vector fields $X, Y$,
    \begin{align*}
        0 = & \: (\mathcal{L}_\xi\widetilde{g}_K)(X,Y)\\
          = & \: \xi \widetilde{g}_K(X,Y)-\widetilde{g}_K([\xi,X],Y)-\widetilde{g}_K(X,[\xi,Y])\\
          = & \: \widetilde{g}_K(\nabla_\xi X, Y)+\widetilde{g}_K(X,\nabla_\xi Y)-\widetilde{g}_K([\xi,X],Y)-\widetilde{g}_K(X,[\xi,Y])\\
          = & \: \widetilde{g}_K(\nabla_X\xi,Y)+\widetilde{g}_K(X,\nabla_Y\xi)\\
          = & \: \widetilde{g}_K(J\nabla_XE,Y)+\widetilde{g}_K(X,J\nabla_YE)\\
          = & \: -\widetilde{g}_K(\nabla_XE,JY)+\widetilde{g}_K(X,\nabla_{JY}E)
    \end{align*}
    By replacing $Y$ with $-JY$, it follows that
    \begin{equation*}
        \widetilde{g}_K(\nabla_XE,Y) = \widetilde{g}_K(X,\nabla_YE).
    \end{equation*}
    A similar computation based on $\mathcal{L}_{E}\widetilde{g}_K=2\widetilde{g}_K$ shows that
    \begin{equation*}
        \widetilde{g}_K(\nabla_XE,Y)+\widetilde{g}_K(X,\nabla_YE)=2\widetilde{g}_K(X,Y).
    \end{equation*}
    Hence 
    \begin{equation}\label{eq:derivative of E}
        \nabla E=\text{Id}.
    \end{equation}
    
    We now construct an explicit Kähler cone biholomorphically isometric to $(\widetilde{M},\widetilde{g}_K)$. Define $F=\frac{1}{2}\widetilde{g}_K(E,E)$. Then $\nabla F=E$ by \eqref{eq:derivative of E}. Because $E$ is a complete vector field and nowhere vanishes, $S=F^{-1}(\frac{1}{2})\subseteq\widetilde{M}$ is a smooth hypersurface. Suppose that $\Psi:\R\times S\to \widetilde{M}$ is the flow generated by $E$. Define a diffeomorphism $\Phi:\R_+\times S\to \widetilde{M}$ as follows:
    \begin{equation}
        \quad \Phi(r,s) = \Psi(\log r, s).
    \end{equation}
    It now remains to check that $(\R_+\times S,\Phi^*J,\Phi^*\widetilde{g}_K)$ is a K\"{a}hler cone. First, it is clearly K\"{a}hler. Moreover, it follows from \eqref{eq:derivative of E} that $\Phi^*\widetilde{g}_K = dr^2+r^2g_S$, where $g_S$ is the restriction of $\widetilde{g}_K$ to $S$.  Finally, $r\frac{\partial}{\partial r}=\Phi^*E$ is real holomorphic. 
\end{proof}

Note that $\Gamma := \pi_1(M)$ acts on this K\"{a}hler cone conformally. Therefore, by Theorem \ref{thm: complete Kahler cone}, the Sasaki link $(S,g_S)$ is complete.   

The completeness implies that $(\widetilde{M}, \widetilde{g}_K)$ has one-point completion $(\widetilde{M}\cup\{o\}, d)$, where $o$ is the cone tip. Let $\chi:\Gamma\to\R_+$ be the homothety character, i.e. $\gamma^*\widetilde{g}_K=\chi(\gamma)\widetilde{g}_K, \forall \gamma\in\Gamma$. Then $r(\gamma\cdot x)^2=d(\gamma\cdot x, o)^2=\chi(\gamma)d(x,o)^2=\chi(\gamma)r(x)^2$. Consequently,
\[\gamma^*g^V=\frac{\gamma^*\widetilde{g}_K}{\gamma^*r^2}=\frac{\widetilde{g}_K}{r^2}=g^V,\quad \forall\gamma\in\Gamma.\]
Hence the Vaisman metric $g^V$ descends to $M = M/\Gamma$, thereby proving alternative (2) of Theorem \ref{thm: lck manifolds with nonnegative bisectional curvature}.

\subsubsection{$\phi$ is trivial} We will prove that (3) in Theorem \ref{thm: lck manifolds with nonnegative bisectional curvature} holds.

\begin{lemma}
    The nontrivial complex Lie group $G$ is the abelian group $\C^k$.
\end{lemma}
\begin{proof}
    Since $\phi$ is trivial, $f$ is $G$--invariant. It follows that $f$ is constant on the orbits of $G$. This shows that restrictions of $\widetilde{g}$ on the orbits are K\"{a}hler. Hence $G$ admits left invariant K\"{a}hler metrics and therefore is isomorphic to $\C^k, k>0$.
\end{proof}

Let $V$ be the distribution of tangent spaces to the orbits and $H=V^\perp$ be its orthogonal complement.
\begin{lemma}
    $H$ is an integrable distribution.
\end{lemma}
\begin{proof}
    Let $\widetilde{\omega}$ be the fundamental form of $\widetilde{g}$ and it has the following decomposition 
    \begin{equation}
        \widetilde{\omega} = \widetilde{\omega}^H+\widetilde{\omega}^V
    \end{equation}
    with respect to the decomposition $T\widetilde{M}=H\oplus V$.
    $e^f\widetilde{\omega}$ is a K\"{a}hler form and hence
    \begin{equation}\label{eq: Kahler form}
        0 = d(e^f\widetilde{\omega}) = d(e^f)\wedge\widetilde{\omega}^H + d(e^f)\wedge\widetilde{\omega}^V + e^f d\widetilde{\omega}^H + e^f d\widetilde{\omega}^V.
    \end{equation}
    Let $\eta_1, \eta_2$ be smooth sections of $H$ and $\xi\in\Lie{g}$. Since $f$ is $G$ invariant, $d(e^f)(\xi)=0$. Furthermore, since $G$ preserves $\widetilde{\omega}^H$, $\mathcal{L}_\xi\widetilde{\omega}^H=0$. As a consequence, $d\widetilde{\omega}^H(\xi,\cdot,\cdot)=0$. Therefore, \eqref{eq: Kahler form} shows that
    \begin{equation}
        \begin{aligned}
            0 =& \: d\widetilde{\omega}^V(\eta_1,\eta_2,\xi)\\
              =& \: \eta_1\widetilde{\omega}^V(\eta_2,\xi) + \eta_2\widetilde{\omega}^V(\xi,\eta_1) + \xi\widetilde{\omega}^V(\eta_1,\eta_2)\\ 
                & - \widetilde{\omega}^V([\eta_1,\eta_2],\xi) - \widetilde{\omega}^V([\eta_2,\xi],\eta_1) - \widetilde{\omega}^V([\xi,\eta_1],\eta_2)\\
              = & \: \widetilde{\omega}^V(\xi,[\eta_1,\eta_2]).
        \end{aligned}
    \end{equation}
    Because the vector fields in $\Lie{g}$ span $V$ at every point, we deduce that $[\eta_1,\eta_2]$ is a smooth section of $H$. By the Frobenius theorem, $H$ is an integrable distribution.
\end{proof}

Let $L$ be a leaf of $H$, which is a complex submanifold. Define the holomorphic map
\begin{equation}
    F: L\times\C^k \to \widetilde{M}, \quad F(x,z) = z\cdot x.
\end{equation}

\begin{lemma}
    $F$ is a biholomorphism.
\end{lemma}
\begin{proof}
    For any leaf $L'$ of the distribution $H$, let $S(L')=\C^k\cdot L'$ be its saturation under the action of $\C^k$. It is clear that the saturation of any leaf is open and saturations of different leaves are either the same or disjoint. Therefore, by the connectedness of $\widetilde{M}$, $S(L)=\widetilde{M}$, i.e. $F$ is surjective.

    Let 
    \[\Gamma_L=\{z\in\C^k:z\cdot L=L\}\leq \C^k\]
    and it acts freely discretely on $L\times\C^k$ by
    \[z\cdot(x,w)=(z\cdot x, w-z), \quad \forall z\in\Gamma_L.\]
    The preimages of points in $\widetilde{M}$ are orbits of $\Gamma_L$--action. Therefore, $F$ induces a biholomorphism from $(L\times\C^k)/\Gamma_L$ to $\widetilde{M}$. Because $\widetilde{M}$ is simply connected, we deduce that $\Gamma_L=1$ and $F$ is a biholomorphism.
\end{proof}

By the above lemma, $\widetilde{g}_K$ is regarded as a K\"{a}hler metric on $L\times\C^k$. The next lemma shows that it splits.
\begin{lemma}
    $\widetilde{g}_K=g_L+g_0$, where $g_L$ is a K\"{a}hler metric on $L$ and $g_0$ is the Euclidean metric on $\C^k$.
\end{lemma}
\begin{proof}
    Let $z=(z^1,\cdots,z^l)$ be local coordinates on $L$ and $w=(w^1,\cdots,w^k)$ be standard coordinates on $\C^k$. Since the two factors are orthogonal, we have the following local expression of the K\"{a}hler form
    \[\widetilde{\omega}_K=\sqrt{-1}\left(g_L(z,\Bar{z},w,\Bar{w})_{i\Bar{j}}dz^i\wedge d\Bar{z}^j+H(z,\Bar{z},w,\Bar{w})_{\alpha\Bar{\beta}}dw^\alpha\wedge d\Bar{w}^\beta\right).\]
    Since $\widetilde{\omega}_K$ is invariant under the action of $\C^k$, $g_L(z,\Bar{z},w,\Bar{w})_{i\Bar{j}}=g_L(z,\Bar{z})_{i\Bar{j}}$ and $H(z,\Bar{z},w,\Bar{w})_{\alpha\Bar{\beta}}=H(z,\Bar{z})_{\alpha\Bar{\beta}}$. Finally, using $d\widetilde{\omega}_K=0$, we get $dH_{\alpha\Bar{\beta}}=0$ and thus $H_{\alpha\Bar{\beta}}$ are constants. The result then follows. 
\end{proof}

We have now shown that $(\widetilde{M},\widetilde{g}_K)\cong(L,g_L)\times(\C^k,g_0)$. Finally, we observe that $g_L$ has to be incomplete. Indeed, if it were complete, then $(\widetilde{M},\widetilde{g}_K)$ would also be complete. Since $M$ does not admit any K\"{a}hler metric, there exists $\gamma\in\pi_1(M)$ such that $\gamma^*\widetilde{g}_K=c\widetilde{g}_K, c < 1$, namely, $\gamma$ is a contraction. Since contractions on complete metric spaces have fixed points, $\gamma$ fixes some point of $\widetilde{M}$. This contradicts the fact that $\Gamma$ acts freely on $\widetilde{M}$.

\subsection{Applications}\label{subsec:applications}

In this subsection we give applications of Theorem \ref{thm: lck manifolds with nonnegative bisectional curvature}. 

\begin{corollary}\label{cor: lck manifolds with pi1=Z}
    Suppose that $M$ is a compact complex manifold and $\pi_1(M)=\Z$. If $M$ admits an l.c.K. metric with nonnegative bisectional curvature, then exactly one of the following holds:
    \begin{enumerate}
        \item $M$ admits a K\"{a}hler metric.
        \item $M$ admits a Vaisman metric.
    \end{enumerate}
\end{corollary}
\begin{proof}
    Suppose that $M$ is non--K\"{a}hler. It then suffices to show that (3) in Theorem \ref{thm: lck manifolds with nonnegative bisectional curvature} cannot occur under the assumptions. Let $\widetilde{g}_K$ be the K\"{a}hler metric on the universal cover and $\chi:\pi_1(M)=\Z\to\R_+$ be its homothety character. Because $M$ is non--K\"{a}hler, $\chi$ is non--trivial and thus injective. This implies that $(\widetilde{M},\widetilde{g}_K)$ is a cone--like space in the sense of Belgun--Moroianu (cf. \cite[Definition 2.3]{BM16}), which has one--point completion by Theorem 1.5 in \cite{BM16}. Therefore, (3) is impossible and $M$ admits a Vaisman metric.
\end{proof}
This yields the following concrete non--existence result.
\begin{corollary}\label{cor: lck metrics on non-diagonal Hopf surfaces}
    Non--diagonal Hopf surfaces do not admit l.c.K. metrics with nonnegative bisectional curvature.
\end{corollary}
\begin{proof}
    Since non--diagonal Hopf surfaces are non--K\"{a}hler and non--Vaisman (cf. \cite[Theorem 1]{Bel00}), this result follows immediately from Corollary \ref{cor: lck manifolds with pi1=Z}.
\end{proof}
\begin{remark}
    As mentioned in the Introduction, it is shown in \cite{Yan17} that a non--K\"{a}hler compact complex surface which admits Hermitian metrics with nonnegative bisectional curvature is a Hopf surface. While there is an explicit construction of such metrics on diagonal Hopf surfaces (cf. \cite{GO98,Yan17}), the existence of such metrics on non--diagonal ones is unknown. With the additional l.c.K. assumption, the above corollary answers this question negatively.
\end{remark}

\subsection{Questions}\label{subsec: questions}

We conclude with two questions concerning possible refinements of Theorem \ref{thm: lck manifolds with nonnegative bisectional curvature}. In alternatives (1) and (2), the proof establishes only the existence of a K\"{a}hler or Vaisman metric and provides no curvature information about that metric. This leads to the following question.

Suppose that $M$ admits an l.c.K. metric with nonnegative Chern bisectional curvature and also admits a K\"{a}hler (respectively, Vaisman) metric. Must $M$ admit a K\"{a}hler (respectively, Vaisman) metric with nonnegative Chern bisectional curvature?

We also do not know whether alternative (3) actually occurs. Inoue surfaces satisfy the geometric conclusion in (3), but their tangent bundles are not nef (cf. \cite[Proposition 6.4]{DPS94}), so they admit no Hermitian metric with nonnegative Chern bisectional curvature. We therefore ask:

Are there compact l.c.K. manifolds which have nonnegative bisectional curvature and satisfy the conditions in case (3)?

\appendix
\section{Sasaki manifolds with positive transverse holomorphic sectional curvature}

The purpose of this appendix is to prove Corollary \ref{cor: compactness}. To that end, we will deduce a sharp diameter estimate for complete Sasaki manifolds whose transverse holomorphic sectional curvature is bounded below by $4$ (cf. Theorem \ref{thm: diameter estimate}). This can be regarded as a generalization of the well--known diameter estimate for K\"{a}hler manifolds with positive holomorphic sectional curvature (cf.\cite{Tsu57,CLZ25}).

Recall the intrinsic definition of transverse curvature tensor (cf.\cite[\S 3]{He13}): Given a complete Sasaki manifold $(S,g_S,J)$, let $D=(\R\xi)^\perp$ be the contact distribution. The transverse connection $\nabla^T$ on $D$ is defined by
\begin{equation}
    \nabla^T_XY = (\nabla_XY)^D, \quad \nabla^T_\xi X=[\xi,X]^D, \quad \forall X,Y\in \Gamma(D)
\end{equation}
where the superscript $D$ denotes the projection onto $D$. The transverse curvature $R^T$ can be defined by
\begin{equation}
    R^T(X,Y)Z = \nabla^T_X\nabla^T_YZ-\nabla_Y^T\nabla_X^TZ-\nabla^T_{[X,Y]}Z,\quad \forall X,Y\in\Gamma(TS), Z\in\Gamma(D).
\end{equation}
The following formula follows from a straightforward calculation:
\begin{lemma}\label{lem: transverse curvature}
    For $X,Y,Z,W\in\Gamma(D)$, 
    \begin{equation}
        \begin{aligned}
            & R^T(X,Y,Z,W) \\
            = & R(X,Y,Z,W) + g_S(\Phi(X),W)g_S(\Phi(Y),Z)-g_S(\Phi(X),Z)g_S(\Phi(Y),W)\\
            & -2g_S(\Phi(X),Y)g_S(\Phi(Z),W),
        \end{aligned}
    \end{equation}
    where $R$ is the curvature tensor of $g_S$ and $\Phi|_D=J, \Phi(\xi)=0$. In particular, 
    \begin{equation}
        R^T(X,\Phi(X),\Phi(X),X) = R(X,\Phi(X),\Phi(X),X)+3g_S(X,X)^2.
    \end{equation}
\end{lemma}

Next, we prove a diameter estimate for complete Sasaki manifolds.
\begin{theorem}\label{thm: diameter estimate}
    If $(S,g_S,J)$ has transverse holomorphic sectional curvature bounded below by $4$, then its diameter is at most $\pi$.
\end{theorem}
\begin{proof}
    We will prove this using the second variation formula for geodesics. Let $\gamma:[0,l]\to S$ be a unit speed minimizing geodesic and $T=\gamma'$ be the velocity field. By the second variation formula, for any vector field $V$ along $\gamma$ such that $V(0)=V(l)=0$,  
    \begin{equation}\label{eq: second variation}
        \int_0^lg_S(V',V')-R(V,T,T,V)dt\ge 0.
    \end{equation}
    
    Because the Reeb field $\xi$ is Killing, $\eta(T)=g(\xi,T)=b$ is a constant function. It is clear that $|b|\leq 1$. We consider two cases.

    If $|b|=1$, then $T=\pm\xi$. Let $X$ be a parallel vector field along $\gamma$ such that $X(t)\in D_{\gamma(t)}$. Then $R(X,T,T,X)=R(X,\xi,\xi,X)\equiv 1$. 
    Let $V(t)=\sin\left(\frac{\pi}{l}t\right)X(t)$ in \eqref{eq: second variation} and we deduce $l\leq \pi$.

    If $|b|<1$, let $a=\sqrt{1-b^2}>0$. Define
    $$E_1=\frac{1}{a}\Phi(T), \quad E_2=\frac{bT-\xi}{a},$$
    then $\nabla_TE_1=E_2, \nabla_TE_2=-E_1$. It follows that 
    $$P(t)=(\cos t) E_1-(\sin t) E_2$$
    is a parallel vector field along $\gamma$. Since the transverse holomorphic sectional curvature is at least $4$, using Lemma \ref{lem: transverse curvature}, we have
    \begin{equation}
        \begin{aligned}
             & R(E_1,T,T,E_1)\\
            = & a^2R(E_1,\Phi(E_1),\Phi(E_1),E_1)+b^2R(E_1,\xi,\xi,E_1)\\
            = & a^2(R^T(E_1,\Phi(E_1),\Phi(E_1),E_1)-3)+b^2\\
            \ge & a^2 + b^2 = 1.
        \end{aligned}
    \end{equation}
    Moreover,
    \begin{equation}
        R(E_2,T,T,E_2) = R(E_1,\xi,\xi,E_1) = 1
    \end{equation}
    and 
    \begin{equation}
        \begin{aligned}
            & R(T,E_1,E_2,T) \\
            = & -\frac{1}{a}R(T,\Phi(T),\xi,T)\\
            = & \frac{1}{a}g_S(\eta(T)\Phi(T),T) = 0,
        \end{aligned}
    \end{equation}
    where in the last line we used the Sasaki identity $R(X,\xi)Y=\eta(Y)X-g_S(X,Y)\xi$. This shows 
    \begin{equation}
        R(P,T,T,P) = (\cos^2t)R(E_1,T,T,E_1) + (\sin^2t)R(E_2,T,T,E_2)\ge 1.
    \end{equation}
    Therefore, by choosing $V(t)=\sin(\frac{\pi}{l}t)P$ in \eqref{eq: second variation}, $l\leq \pi$.
\end{proof}
\begin{remark}
    The diameter estimate is sharp. The unit sphere with standard Sasaki structure has constant transverse holomorphic sectional curvature $4$ and its diameter is $\pi$.
\end{remark}

\begin{corollary}\label{cor: compactness}
    Let $(S,g_S,J)$ be a complete Sasaki manifold. If its transverse holomorphic sectional curvature is bounded below by a constant $\kappa>0$, then it is compact.
\end{corollary}
\begin{proof}
    The $\mathcal{D}$ homothety rescales the transverse K\"{a}hler metric and hence the transverse holomorphic sectional curvature. Therefore, we may apply a $\mathcal{D}$ homothety to the Sasaki structure so that its holomorphic sectional curvature is at least $4$.
\end{proof}

\section{K\"{a}hler cones with cocompact conformal group actions}
In this appendix we prove the following completeness result on K\"{a}hler cones.
\begin{theorem}\label{thm: complete Kahler cone}
    Let $(C(S), J, dr^2+r^2g_S)$ be a simply connected K\"{a}hler cone of complex dimension $n\ge 2$, without assuming completeness of $(S,g_S)$.
    If there is a group $\Gamma$ acting conformally such that $C(S)/\Gamma$ is a compact manifold, then $(S,g_S)$ is complete.
\end{theorem}
\begin{proof}
    The K\"{a}hler cone metric endows a similarity structure on the quotient manifold $C(S)/\Gamma$. Therefore, by the structure theorem of closed similar manifolds (cf. \cite[Theorem 1.5]{Kou19}), one of the following holds:
    \begin{enumerate}
        \item [i.] $C(S)$ is flat.
        \item [ii.] $C(S)$ is Riemannian irreducible.
        \item [iii.] $C(S)$ is isometric to $(\R^q,g_0)\times (N,g_N), q>0$, where the first factor is a Euclidean space and the second factor is an incomplete Riemannian irreducible manifold.
    \end{enumerate}
    
    If i. holds, by \cite[Theorem 2]{Fri80}, $C(S)$ is isometric to $\R^{2n}\setminus\{0\}$ and hence has complete Sasaki link. 

    Suppose that ii. is true. Let $E$ be the homothetic vector field. Then $\nabla(\gamma_*E)=\text{Id}$ and $\gamma_*E-E$ is a parallel vector field for any $\gamma\in\Gamma$. Because $C(S)$ is irreducible, $\gamma_*E=E, \forall\gamma\in\Gamma$. This shows that $\gamma$ preserves the flow lines of $E$ and consequently, $\gamma$ extends to the cone with tip $\widehat{C(S)}=C(S)\cup\{o\}$ and $\gamma(o)=o$. Repeating the argument at the end of Subsection \ref{subsubsection: non-trivial} shows that $\Gamma$ preserves the cylindrical metric $\frac{1}{r^2}dr^2+g_S$, which descends to the compact quotient $C(S)/\Gamma$. Therefore, $(S,g_S)$ is complete.

    Finally, we show that iii. is impossible. Indeed, assume that iii. held, then consider the decomposition $E= E_{\R^q}+E_N$ of $E$ onto the two factors. $E_N$ is complete and $\nabla E_N = \text{Id}_{N}$. Therefore, $(N,g_N)$ is also a Riemannian cone with homothetic vector field $E_N$. Suppose $(N,g_N)=(\R_+\times T, ds^2+s^2g_T)$. Since $(N,g_N)$ is irreducible, the arguments in case ii. and Subsection \ref{subsubsection: non-trivial} imply that $\Gamma$ preserves 
    \[h = \frac{1}{s^2}(g_0+g_N) = e^{-2u}g_0+du^2+g_T,\quad u=\log s.\]
    Observe that $\nabla^hu$ is invariant under $\Gamma$ and hence descends to $C(S)/\Gamma$. Therefore, 
    \[\int_{C(S)/\Gamma}\Div(\nabla^hu)=0,\]
    which contradicts $\Div(\nabla^hu)=-q<0$.
\end{proof}

\subsection*{Acknowledgements} 
The author is grateful to Professor Lei Ni for his guidance and helpful
discussions concerning this work. The author is especially grateful to
Professor Anna Fino for her careful reading of an earlier version of the
entire manuscript and for numerous detailed comments and suggestions,
which substantially improved both its content and exposition. The author
also thanks Professors Andrei Moroianu and Misha Verbitsky for valuable
comments.

\bibliographystyle{alpha}
\bibliography{refs}

\end{document}